\documentclass[12pt]{amsart}

\usepackage{amssymb,amsthm,mathtools,bm,color,scalerel}
\usepackage{hyperref}

\def\R{{\mathbb{R}}}
\def\P{{\mathcal{P}}}
\def\zero{{\mathbf{0}}}
\def\aa{{\bm{a}}}
\def\xx{{\bm{x}}}
\def\yy{{\bm{y}}}
\def\zz{{\bm{z}}}
\def\phi{\varphi}

\newtheorem{thm}{Theorem}
\newtheorem*{conj}{Conjecture}
\newtheorem{lemma}{Lemma}
\newtheorem{claim}{Claim}
\newtheorem*{proposition}{Proposition}
\newtheorem*{prob}{Problem}

\usepackage{tikz}
\usetikzlibrary{positioning}
\usetikzlibrary{svg.path}
\definecolor{orcid_color}{HTML}{A6CE39}

\hypersetup{pdfborder={0 0 0}} 
\DeclareRobustCommand{\orcidicon}{%
	\raisebox{.2mm}{\scalerel*{%
	\begin{tikzpicture}[xscale=1,yscale=-1,transform shape]
	\filldraw[color=orcid_color] svg {M256,128c0,70.7-57.3,128-128,128C57.3,256,0,198.7,0,128C0,57.3,57.3,0,128,0C198.7,0,256,57.3,256,128z};
	\filldraw[color=white] svg {M86.3,186.2H70.9V79.1h15.4v48.4V186.2z} svg {M108.9,79.1h41.6c39.6,0,57,28.3,57,53.6c0,27.5-21.5,53.6-56.8,53.6h-41.8V79.1z M124.3,172.4h24.5
		c34.9,0,42.9-26.5,42.9-39.7c0-21.5-13.7-39.7-43.7-39.7h-23.7V172.4z} svg {M88.7,56.8c0,5.5-4.5,10.1-10.1,10.1c-5.6,0-10.1-4.6-10.1-10.1c0-5.6,4.5-10.1,10.1-10.1
		C84.2,46.7,88.7,51.3,88.7,56.8z};
	\end{tikzpicture}}{|}}%
}
\newcommand{\orcid}[1]{\href{https://orcid.org/#1}{\orcidicon}}

\title[Alon's transmitting problem]{Alon's transmitting problem and\\ multicolor Beck--Spencer Lemma}
\author[Norihide Tokushige]{Norihide Tokushige\,\orcid{0000-0002-9487-7545}}
\address{College of Education, University of the Ryukyus, Nishihara  903-0213, Japan}
\email{hide@cs.u-ryukyu.ac.jp}
\urladdr{http://www.cc.u-ryukyu.ac.jp/~hide/}

\begin{document}
\begin{abstract}
The Hamming graph $H(n,q)$ is defined on the vertex set $\{1,2,\ldots,q\}^n$ and
two vertices are adjacent if and only if they differ in precisely one 
coordinate.
Alon (1992) proved that for any sequence $v_1,\ldots,v_b$ of
$b=\lceil\frac n2\rceil$ vertices of $H(n,2)$, there is a vertex 
whose distance from $v_i$ is at least $b-i+1$ for all $1\leq i\leq b$.
In this note, we prove that for any $q\geq 3$ and any sequence 
$v_1,\ldots,v_b$ of
$b=\lfloor(1-\frac1q)n\rfloor$ vertices of $H(n,q)$, there is a vertex 
whose distance from $v_i$ is at least $b-i+1$ for all $1\leq i\leq b$.

Alon used a lemma due to Beck and Spencer (1983) which,
in turn, was based on the floating variable method introduced by 
Beck and Fiala (1981) who studied combinatorial discrepancies. 
For our proof, we extend the Beck--Spencer Lemma by using a multicolor 
version of the floating variable method due to Doerr and Srivastav (2003).
\end{abstract}

\maketitle

\hypersetup{pdfborder={0 0 1}} 

\section{Introduction}
Alon posed a transmitting problem in \cite{Alon} and obtained an optimal
solution. This problem can be described in terms of the graph burning number
(see, e.g., \cite{Bonato}).
Let $G$ be a finite graph with the vertex set $V$. 
For vertices $u,v\in V$ let $d(u,v)$ denote the distance between $u$ and $v$.
For a non-negative integer $k$, 
let $\Gamma_k(v)$ denote the $k$-neighbors of $v$, that is,
the set of vertices $u\in V$ such that $d(u,v)\leq k$.
For $v_1,v_2,\ldots,v_b\in V$ we say that $(v_1,v_2,\ldots,v_b)$
is a burning sequence of length $b$
if $\Gamma_{b-1}(v_1)\cup\Gamma_{b-2}(v_2)\cup\cdots\cup\Gamma_{0}(v_b)=V$. 
The burning number of $G$, denoted by $b(G)$, is defined to be the minimum 
length of the burning sequences. We can think of the vertices as processors.
Suppose that there is a sender outside the graph, and it sends a message
to $v_i$ at round $i$. Each processor that receives a message at round $j$
transmits the message to all its neighbors at round $j+1$. Then $b(G)$ is the
minimum number of rounds in which all the processors share the message.

\begin{thm}[Alon \cite{Alon}]\label{thm:Alon} 
Let $G$ be the $n$-dimensional cube.
Then $b(G)=\lceil \frac n2\rceil+1$.
\end{thm}
For the $n$-dimensional cube, let $v_1$ be an arbitrary vertex, and let 
$v_2$ be the antipodal vertex. Then it is easy to check that 
$\Gamma_{b-1}(v_1)\cup\Gamma_{b-2}(v_2)$ covers all the vertices where 
$b=\lceil \frac n2\rceil+1$. 
This means that the burning number is at most $b$. The more difficult part
is to show that no matter how $b-1$ vertices are chosen, they cannot be a 
burning sequence, that is, after $b-1$ rounds there is still some vertex 
that has not received the message yet.

In this note, we extend the above result to Hamming graphs.
For positive integers $n,q$, 
let $[q]:=\{1,2,\ldots,q\}$ and let $[q]^n$ denote the set of 
$n$-tuple of elements of $[q]$. The Hamming distance between $u,v\in[q]^n$ 
is defined to be the number of entries (coordinates) that they differ.
The Hamming graph $H(n,q)$ has the vertex set $V=[q]^n$, and
two vertices are adjacent if they have Hamming distance one. 
Note that $H(n,2)$ is (isomorphic to) the $n$-dimensional cube. 
Our main result is the following.
\begin{thm}\label{thm:main}
Let $q\geq 3$ be an integer, and let $G=H(n,q)$ be the Hamming graph. Then
\[
\left\lfloor\left(1-\tfrac1q\right)n\right\rfloor +1\leq b(G)
\leq\left\lfloor\left(1-\tfrac1q\right)n+\tfrac{q+1}2\right\rfloor.
\]
\end{thm}
The upper bound in Theorem~\ref{thm:main} is easily verified by construction.
Indeed, the sequence $v_i=(i,i,\ldots,i)$ for $i=1,2,\ldots,q$ 
(and any $v_i$ for $i>q$) works as a corresponding burning sequence 
(see \cite{T2024} for more details).
This construction is valid for $q=2$ as well, and in this case the 
upper bound coincides with the correct value of $b(H(n,2))$ given by 
Theorem~\ref{thm:Alon}.

To show the lower bound in Theorem~\ref{thm:Alon}, Alon used the
Beck--Spencer lemma.
\begin{lemma}[Beck--Spencer \cite{Beck-Spencer}]\label{lemma:BS}
For $1\leq i\leq n$, let $\aa_i\in\{-1,1\}^n$ be given.
Then there exists $\xx\in\{-1,1\}^n$ such that the standard inner product
satisfies $|\aa_i\cdot\xx|<2i$ for all $1\leq i\leq n$.
\end{lemma}

We can think of $\{-1,1\}^n$ as the vertex set of the $n$-dimensional
cube, and it follows that $\aa\cdot\xx = n-2d(\aa,\xx)$ as we will see
in the next section. Then Lemma~\ref{lemma:BS} is restated as follows.

\begin{lemma}[Lemma~\ref{lemma:BS} restated]\label{lemma:BS2}
For $1\leq i\leq n$, let $v_i$ be a given vertex of $H(n,2)$.
Then there exists a vertex $w$ such that 
{\upshape $|n-2d(v_i,w)|<2i$} for all $1\leq i\leq n$.
\end{lemma}

For our proof of the lower bound in Theorem~\ref{thm:main}, we need a
multicolor version of the Beck--Spencer lemma.

\begin{lemma}\label{lemma:dist}
Let $q\geq 3$.
For $1\leq i\leq n$, let $v_i$ be a given vertex of $H(n,q)$.
Then there exists a vertex $w$ such that 
{\upshape $|(1-\frac1q)n-d(v_i,w)|<i$} for all $1\leq i\leq n$.
\end{lemma}

To prove Lemma~\ref{lemma:BS}, Beck and Spencer used the so-called 
\emph{floating variable method} based on \cite{Beck-Fiala} where
Beck and Fiala studied combinatorial discrepancies.
Then Doerr and Srivastav in \cite{Doerr-Srivastav} extended the results
in \cite{Beck-Fiala} (and many other related results concerning 
discrepancies) to multicolor settings.
We utilize their ideas on vector-coloring to prove Lemma~\ref{lemma:dist},
c.f.~the proof of Theorem~4.2 in \cite{Doerr-Srivastav}. 
So we explain a multicolor version of the floating variable method 
in the next section.

It follows from Lemma~\ref{lemma:BS2} that $b(H(n,2))\geq\lfloor\frac n2\rfloor+1$.
However, this is not the sharp bound if $n$ is odd. Alon made one more
twist to improve the bound in \cite{Alon}. The author was not able to
find an appropriate extension of this tricky part for $H(n,q)$ ($q\geq 3$),
and this suggests that the lower bound in Theorem~\ref{thm:main} could
be improved. 

\begin{conj}
For fixed $q\geq 3$ and large enough $n$, we have
$b(H(n,q))=\left\lfloor\left(1-\tfrac1q\right)n+\tfrac{q+1}2\right\rfloor$.
\end{conj}
Since $H(n,q)$ has diameter $n$, we have a trivial bound $b(H(n,q))\leq n+1$.
Thus if the conjecture holds, then we need 
$n\geq\lfloor\frac{q(q-1)}2\rfloor$.
On the other hand, it follows from the lower bound in Theorem~\ref{thm:main}
that $b(H(n,q))=n+1$ if $q\geq n\geq 3$, see the last section.

\section{Proofs}

We want to restate Lemma~\ref{lemma:dist} in the form of 
Lemma~\ref{lemma:BS}. For this, we need to assign a vector to a vertex
in $H(n,q)$ so that the vector behaves nicely with respect to the inner
product and the Hamming distance.
To this end, we use vector-coloring introduced by Doerr and Srivastav in
\cite{Doerr-Srivastav}.
For a $qn$-dimensional vector $\zz=(z_1,\ldots,z_{qn})\in\R^{qn}$, let
$\bar z_j:=(z_{q(j-1)+1},z_{q(j-1)+2},\ldots,z_{qj})\in\R^q$, and we also write
$\zz=(\bar z_1,\ldots,\bar z_n)$. (We write $\bar z$ to indicate the vector
is $q$-dimensional.)
For each color $i\in[q]$, we assign a vector 
$\bar c_i=-\frac1q\bar 1+\bar e_i\in\R^q$,
that is, $\bar c_i$ has $1-\frac1q$ at the $i$-th entry, and $-\frac1q$
at the other $q-1$ entries. 
For example, if $q=3$ then
$\bar c_1=(\frac23,-\frac13,-\frac13)$,
$\bar c_2=(-\frac13,\frac23,-\frac13)$,
$\bar c_3=(-\frac13,-\frac13,\frac23)$.
By definition the sum of all entries of $\bar c_i$ equals $0$. 
We also have
\begin{equation}\label{eq1}
\bar c_i\cdot\bar c_j=
\begin{cases}
1-\frac1q&\text{if $i=j$},\\
-\frac1q&\text{if $i\neq j$}.
\end{cases}
\end{equation}

Let $Q$ be the
set of color vectors, that is,
$Q=\{\bar c_1,\bar c_2,\ldots,\bar c_q\}$.
For each vertex $v=(\nu_1,\ldots,\nu_n)\in[q]^n$ of $H(n,q)$, we assign a vector 
$\zz=(\bar z_1,\ldots,\bar z_n)\in\R^{qn}$ by 
\begin{equation}\label{def:zi}
\bar z_i:=\bar c_{\nu_i}  
\end{equation}
for $1\leq i\leq n$. In this case we understand $\zz\in Q^n$.
Define a bijection $\phi:[q]^n\to Q^n$ by \eqref{def:zi} and
\begin{equation}\label{def:phi}
 \phi(v):=\zz.
\end{equation}

\begin{claim}\label{claim2}
Let $v,w\in[q]^n$, and $\aa=\phi(v),\xx=\phi(w)$.
Then $\aa\cdot\xx=(1-\frac1q)n-d(v,w)$. 
\end{claim}

\begin{proof}
Let $s:=\#\{i\in[n]:\bar a_i\neq\bar x_i\}=d(v,w)$. 
It follows from \eqref{eq1} that
$\aa\cdot\xx =(1-\tfrac1q)(n-s)-\tfrac1qs
=(1-\tfrac1q)n-d(v,w)$.
\end{proof}

Now we can restate Lemma~\ref{lemma:dist}.

\begin{lemma}[Lemma~\ref{lemma:dist} restated]\label{lemma:main}
For $1\leq i\leq n$, let $\aa_i\in Q^n$ be given. 
Then there exists $\xx\in Q^n$ such that $|\aa_i\cdot\xx|<i$
for all $1\leq i\leq n$.
\end{lemma}

\begin{proof}
Let $\tilde Q\subset\R^q$ denote the convex hull of $Q$, that is,
\[
\tilde Q:=\left\{\sum_{i=1}^q\lambda_i\bar c_i:
0\leq\lambda_i\leq 1\,(i\in[q]),\,\sum_{i=1}^q\lambda_i\leq 1\right\}. 
\]
Note that if $\bar x\in\tilde Q$ then the number of entries of $\bar x$
with value $1-\frac1q$ is at most one.

\begin{claim}\label{claim1}
For $\bar a,\bar x\in Q$ and $\bar y\in\tilde Q$, we have the following.
\begin{itemize}
\item $\bar a\cdot\bar x\in\{-\frac1q,1-\frac1q\}$.
\item $-\frac1q\leq \bar a\cdot\bar y\leq 1-\frac1q$.
If $\bar a\cdot\bar y=1-\frac1q$, then $\bar y=\bar a$.
\item $|\bar a\cdot\bar x-\bar a\cdot\bar y|\leq 1$.
Moreover, if equality holds, then (i) $\bar y\in Q$ or 
(ii) there exists $i$ such that $\bar a=\bar x=\bar c_i$ and the $i$-th
entry of $\bar y$ is $-\frac1q$. 
\end{itemize}
\end{claim}

\begin{proof}
Without loss of generality we may assume that $\bar a=\bar c_1$.
The first item follows from \eqref{eq1}.

For the second item, let $\bar y=\sum_{i=1}^q\lambda_i\bar c_i\in\tilde Q$. 
Noting that $\sum_{i=2}^q\lambda_i\leq 1-\lambda_1$ and
\begin{equation*}
 \bar a\cdot\bar y=\bar c_1\cdot\sum_{i=1}^q\lambda_i\bar c_i
=(1-\tfrac1q)\lambda_1-\tfrac1q(\lambda_2+\cdots+\lambda_n),
\end{equation*}
we have $-\frac1q\leq \bar a\cdot\bar y\leq 1-\frac1q$, because
\begin{equation}\label{eq2}
1-\tfrac1q\geq (1-\tfrac1q)\lambda_1\geq 
\bar a\cdot\bar y\geq
(1-\tfrac1q)\lambda_1-\tfrac1q(1-\lambda_1)=\lambda_1-\tfrac1q\geq
-\tfrac1q.
\end{equation}
If $\bar a\cdot \bar y=1-\frac1q$, then $\lambda_1=1$, that is,
$\bar y=\bar c_1$.

By the first and second items we have 
$|\bar a\cdot\bar x-\bar a\cdot\bar y|\leq 1$.
If equality holds, then $(\bar a\cdot \bar x,\bar a\cdot\bar y)=
(-\frac1q,1-\frac1q)$ or $(1-\frac1q,-\frac1q)$.
For the former case, we have $\bar y=\bar a\in Q$ by the second item.
For the latter case, we have $\lambda_1=0$ and $\sum_{i=2}^q\lambda_i=1$
by \eqref{eq2}, and the first entry of $\bar y$ is 
$\lambda_1(1-\frac1q)+\sum_{i=2}^q\lambda_i(-\frac1q)=0-\frac1q(1-\lambda_1)
=-\frac1q$.
\end{proof}

Let $\xx=(x_1,\ldots,x_{qn})=(\bar x_1,\ldots,\bar x_{n})\in\R^{qn}$ be a variable vector.
We describe an algorithm updating $\xx$ step by step in such a way that 
at each step $\xx\in\tilde Q^{n}$, and in the end $\xx\in Q^{n}$.
We call a variable $x_k$ \emph{floating} if $-\frac1q<x_k<1-\frac1q$, 
and \emph{fixed} if $x_k\in\{-\frac1q,1-\frac1q\}$. 
Once a variable is fixed, then it stays the same value and it is treated as a constant.

Before going into detail, we explain the algorithm's geometric meaning.
Let $\P$ be the polytope obtained by taking the intersection of $\tilde Q^n$
and $n$ hyperplanes $\sum_{i=1}^q x_{q(j-1)+i}=0$ ($1\leq j\leq n$),
cf.~\cite{Godsil}.
Starting from the origin, we move along a line until we hit a facet of 
$\P$, and let $\xx_1$ be the intersection of the line and the facet.
Then, in the facet, starting from $\xx_1$, we move along another line until
we hit a face of $\P$, and let $\xx_2$ be the intersection.
We repeat this procedure so that the dimension of the faces we hit is 
decreasing. Finally, we stop at a face of dimension $0$, and we find
one of the vertices $\xx$ of $\P$, that is, $\xx\in Q^n$.
The choice of lines is crucial for this algorithm to satisfy $|\aa_i\cdot\xx|<i$ for all $i$.
To make the right choice, we will solve a system of equations repeatedly as 
described below.

We consider the following three conditions.
\begin{itemize}
 \item[(C1) ] $\aa_i\cdot\xx=0$,
 \item[(C2) ] $x_{q(j-1)+1}+x_{q(j-1)+2}+\cdots+x_{qj}=0$,
 \item[(C3) ] $\xx\in\tilde Q^{n}$.
\end{itemize}
\noindent
We will repeatedly solve a system of equations satisfying these conditions 
according to the following procedure.

\begin{description}
\item[\textbf{Step(0)} ]
Consider the following system Eq(0) of equations: 
\begin{itemize}
\item (C1) for $1\leq i\leq n$, 
\item (C2) for $1\leq j\leq n$.
\end{itemize}
Then $\xx_0:=\zero$ is a solution to Eq(0), and it satisfies (C3) as well
because $\zero\in\R^q$ is the center of $\tilde Q$, and $\zero\in\R^{rq}$ is inside $\tilde Q^n$.

\item[\textbf{Step(1)} ]
Consider the following system Eq(1) of equations:
\begin{itemize}
\item (C1) for $1\leq i\leq n-1$, 
\item (C2) for $1\leq j\leq n$.
\end{itemize}
Eq(1) has $qn$ variables and $2n-1$ equations. Thus we have a non-trivial
solution $\yy\neq\zero$. For every $\lambda\geq 0$, $\lambda\yy$ also satisfies
Eq(1). Moreover, if $\lambda=0$, then it satisfies (C3) as well.

Increase $\lambda$ continuously starting from $0$.
Then at some point at least one of the entries of $\lambda\yy$ becomes
$-\frac1q$ or $1-\frac1q$ for the first time. Using this $\lambda>0$, let
$\xx_1:=\xx_0+\lambda\yy$. Then $\xx_1$ satisfies Eq(1) and (C3).
This $\xx_1$ is not yet determined and may be updated.
Let $s=1$ and go to \textbf{Step(s)} below.

\item[\textbf{Step(s)} ]
We have a temporary vector $\xx_{s}=(\bar x_1,\ldots,\bar x_{n})$ which satisfies
(C1) for $1\leq i\leq n-s$, (C2) for $1\leq j\leq n$, and (C3). 
For $r\in\{0,1,\ldots,q\}$ let
\[
 F_r(\xx_s):=\{j\in[n]:\text{the number of fixed entries in $\bar x_j$ is $r$}\}. 
\]
We claim that $F_{q-1}(\xx_s)=\emptyset$. Suppose, to the contrary, that
$j\in F_{q-1}(\xx_s)$. Recall that the entries of $\bar x_j$ can contain
$1-\frac1q$ at most once. If $\bar x_j$ contains $1-\frac1q$, then the
other $q-2$ fixed entries are $-\frac1q$. Then, by (C2), the remaining
entry is $-\frac1q$, a contradiction.
If $\bar x_j$ does not contain $1-\frac1q$, then all the $q-1$ fixed 
entries are $-\frac1q$, and the remaining entry is $1-\frac1q$, a contradiction.

If $|F_q(\xx_s)|\geq s$, then we determine $\xx_s$ and complete Step($s$).
Moreover, if $s=n$, then exit the procedure, otherwise
let $\xx_{s+1}:=\xx_s$ and proceed to Step($s+1$).
Note that $\xx_{s+1}$ satisfies (C1) for $1\leq i\leq n-(s+1)$, 
(C2) for $1\leq j\leq n$, and (C3).

From now on, we deal with the case $|F_q(\xx_s)|\leq s-1$.
Since $|F_q(\xx_s)|$ is non-decreasing in $s$, and $|F_q(\xx_{s-1})|\geq s-1$
holds when $\xx_{s-1}$ is determined, we may assume that $|F_q(\xx_s)|=s-1$.
Let $f_r=|F_r(\xx_s)|$. 
Since $f_{q-1}=0$ and $\sum_{r=0}^{q-2} f_r=n-f_q$, it follows that 
$\sum_{r=1}^{q-2}rf_r\leq (q-2)\sum_{r=1}^{q-2}f_r=(q-2)(n-f_q)$.

Recall that we treat a fixed variable as a constant.
Let $[qn]=I_c\cup I_v$ be a partition, where $I_c$ is the set of indices
for constant entries of $\xx_s=(x_1,x_2,\ldots,x_{qn})$, 
and $I_v$ is the set of indices for floating
variables of $\xx_s$. Then $|I_v|=qn-\sum_{r=1}^q rf_r$.

Consider the following system Eq($s$) of equations:
\begin{itemize}
\item (C1) for $1\leq i\leq n-s=n-f_q-1$, 
\item (C2) for $j\in F_0(\xx_s)\cup F_1(\xx_s)\cup\cdots\cup F_{q-2}(\xx_s)$.
\end{itemize}
We have $|I_v|$ variables, and $(n-f_q-1)+(n-f_q)=2n-2f_q-1$ equations. 
Thus the dimension of the solution space is
\[
|I_v| - (2n-2f_q-1) = (q-2)(n-f_q)-\sum_{r=1}^{q-2} rf_r +1 \geq 1,
\]
and we obtain a line of solution $\xx_s+\lambda\yy$, where $\lambda\in\R$
and $y_i=0$ for $i\in I_c$.
By continuously increasing $\lambda$ from $0$, we get $\xx_s+\lambda\yy$
such that one of the entries in $I_v$ becomes $-\frac1q$ or $1-\frac1q$ 
for the first time. Using this $\lambda>0$, let $\xx:=\xx_{s}+\lambda\yy$.
Then $\xx$ satisfies (C1) for $1\leq i\leq n-s$, (C2) for $1\leq j\leq n$, 
and (C3).
Update $\xx_{s}:=\xx$, and return to the beginning of \textbf{Step(s)}.

\item[\textbf{Completion of procedure} ]
When \textbf{Step(n)} is completed, we obtain $\xx_1,\ldots,\xx_{n}$.
By construction, we have the following conditions.
\begin{itemize}
\item $|F_q(\xx_s)|\geq s$ and $\xx_s\in\tilde Q^n$ for $1\leq s<n$. 
\item $|F_q(\xx_n)|=n$, that is, $\xx_n\in Q^n$.
\item $\aa_i\cdot\xx_{n-i}=0$ for $1\leq i\leq n$.
(Indeed $\aa_i\cdot\xx_s=0$ for $1\leq s<n$, $1\leq i\leq n-s$.)
\end{itemize}
\end{description}

By running this algorithm, determine $\xx_1,\ldots,\xx_n$, and 
let $\xx:=\xx_n=(\bar x_1,\ldots,\bar x_n)\in Q^n$.

We have $\aa_1\cdot\xx_{n-1}=0$.
If $\xx=\xx_{n-1}$, then $|\aa_1\cdot\xx|=0$.
If $\xx\neq\xx_{n-1}$, then $|F_q(\xx_{n-1})|=n-1$, and 
there exists precisely one $j$ such that $j\not\in F_q(\xx_{n-1})$.
Then, writing $\aa_1=(\bar a_1,\ldots,\bar a_n)$ and
$\xx_{n-1}=(\bar y_1,\ldots,\bar y_n)$, we have $\bar a_j,\bar x_j\in Q$,
$\bar y_j\in\tilde Q\setminus Q$, and by Claim~\ref{claim1}, 
\[
|\aa_1\cdot\xx|=|\aa_1\cdot\xx-\aa_1\cdot\xx_{n-1}|
=|\bar a_j\cdot\bar x_j-\bar a_j\cdot\bar y_j|\leq 1.
\]
Moreover, if $|\bar a_j\cdot\bar x_j-\bar a_j\cdot\bar y_j|=1$, 
then there exists $i$ such that $\bar a_j=\bar x_j=\bar c_i$, 
and the $i$-th entry of $\bar y_j$ is $-\frac1q$. But then the $i$-th
entry of $\bar y_j$ is fixed and remains unchanged thereafter. 
This contradicts the fact that the $i$-th entry of $\bar x_j=\bar c_i$ is 
$1-\frac1q$. Consequently we have 
$|\aa_1\cdot\xx|=|\bar a_j\cdot\bar x_j-\bar a_j\cdot\bar y_j|<1$. 

Similarly, for $1\leq i\leq n$, we have
$\aa_i\cdot\xx_{n-i}=0$ and $|F_q(\xx_{n-i})|\geq n-i$. Thus, letting
$J=[n]\setminus F_q(\xx_{n-i})$, we have $|J|\leq i$ and
\[
|\aa_i\cdot\xx|=|\aa_i\cdot\xx-\aa_i\cdot\xx_{n-i}|
=\left|\sum_{j\in J}(\bar a_j\cdot\bar x_j-\bar a_j\cdot\bar y_j)\right|
\leq\sum_{j\in J}\left|\bar a_j\cdot\bar x_j-\bar a_j\cdot\bar y_j\right|
<i,
\]
where we write $\aa_i=(\bar a_1,\ldots,\bar a_n)$ and
$\xx_{n-i}=(\bar y_1,\ldots,\bar y_n)$.
This completes the proof of Lemma~\ref{lemma:main}.
\end{proof}

\begin{proof}[Proof of Theorem~\ref{thm:main}]
As stated in the discussion following the statement of Theorem~\ref{thm:main}, the upper bound is proved in \cite{T2024}. 
Here we prove the lower bound.
Let $m:=\lfloor(1-\tfrac1q)n\rfloor$.
For arbitrary $m$ vertices 
$v_1,v_2,\ldots,v_m\in [q]^n$,
we show that there exists a vertex 
$w\in [q]^n$ such that $d(v_i,w)\geq m+1-i$ for all $1\leq i\leq m$.

Recall the definition of $\phi$ from \eqref{def:phi}.
Let $\aa_i=\phi(v_i)\in Q^n$ for $1\leq i\leq n$.
By Lemma~\ref{lemma:main} and Claim~\ref{claim2}, we get $\xx\in Q^n$
such that 
\[
|\aa_i\cdot\xx|=\left|(1-\tfrac1q)n-d(v_i,w) \right|<i
\]
for all $1\leq i\leq n$, where $w=\phi^{-1}(\xx)$.
Thus we have 
\[
d(v_i,w)>(1-\tfrac1q)n-i 
\]
for all $1\leq i\leq n$.
Let $n=qk+r$, $0\leq r<q$. Then,
\[
(1-\tfrac1q)n=(q-1)k+r-\tfrac rq.
\]
This together with the fact that the distance is an integer yields
\begin{align*}
d(v_i,w)&\geq
\begin{cases}
 (q-1)k-i+1&\text{if $r=0$},\\
 (q-1)k+r-i&\text{if $1\leq r<q$}
\end{cases}\\
&=m+1-i.
\end{align*}
Thus $w\not\in\Gamma_{m-1}(v_1)\cup\cdots\cup\Gamma_0(v_m)$, 
which means that $b(H(n,q))\geq m+1$.
\end{proof}

\section{Concluding remarks}
In this subsection, we also write $b(n,q)$ to mean $b(H(n,q))$.

\subsection{The case $q\geq n$}
Recall that $H(n,q)$ has diameter $n$, so $b(n,q)\leq n+1$ for all $n,q$.

\begin{claim}\label{claim3}
Let $q\geq 3$ and $s\leq n$. If $b(n,q-1)\geq s$ then $b(n,q)\geq s+1$.
\end{claim}
\begin{proof}
Suppose that $b(n,q-1)\geq s$ but $b(n,q)\leq s$.
Then we have a burning sequence for $H(n,q)$ of length $s$, 
starting from, without loss of generality, $v_1=q\mathbf{1}$.
Since $[q]^n\setminus\Gamma_{s-1}(v_1)\supset[q-1]^n$,
we need to cover $H(n,q-1)$ in less than $s$ rounds, which is impossible
because $b(n,q-1)\geq s$.
\end{proof}

\begin{proposition}
If $q\geq n\geq 3$ then $b(H(n,q))=n+1$.
\end{proposition}
\begin{proof}
Since $b(n,q+1)\geq b(n,q)$, it suffices to show that $b(n,n)=n+1$.
This is true if $n=3$. Indeed we have $b(3,2)=3$ by Theorem~\ref{thm:Alon},
and so $b(3,3)\geq 4$ by Claim~\ref{claim3}.

Now let $n\geq 4$.
By Theorem~\ref{thm:main} we have
\begin{equation}\label{b(n,n-2)}
 b(n,n-2)\geq\lfloor\left(1-\tfrac1{n-2}\right)n\rfloor+1=
\lfloor n-1-\tfrac2{n-2}\rfloor+1=n-1.
\end{equation}
Applying Claim~\ref{claim3} to \eqref{b(n,n-2)} twice, we get
$b(n,n)\geq n+1$.
\end{proof}

\subsection{Possible extension of Lemma~\ref{lemma:dist}}
For the proof of Theorem~\ref{thm:Alon}, Alon observed the following.
In Lemma~\ref{lemma:BS} we have 
$|\aa_i\cdot\xx|=|\aa_i\cdot(-\xx)|$, and so the inequality on $d(v_i,w)$
in Lemma~\ref{lemma:BS2} also holds if we replace $v_i$ with its 
antipodal vertex. 
It would be nice if a similar extension was possible for 
Lemma~\ref{lemma:dist}, which would improve Theorem~\ref{thm:main}.
For simplicity, let $q=3$ and $n=3k+1\geq 4$ be fixed.
For $v=(x_1,\ldots,x_n)\in\{0,1,2\}^n$, let
$v'=(x_1+1,\ldots,x_n+1)$ and $v''=(x_1+2,\ldots,x_n+2)$, where addition
is done in modulo 3. 
For vertices $u$ and $v$, let $f(u,v):=|\frac23(3k+1)-d(u,v)|$,
and $g(u,v):=\max\{f(u,v),f(u,v'),f(u,v'')\}$.

\begin{prob}
For $1\leq i\leq 3k+1$, 
let $u_i$ be a given vertex of $H(3k+1,3)$ on $\{0,1,2\}^n$.
Then is it true that there exists a vertex $w$ such that 
$g(u_i,w)<i$ for all $1\leq i\leq 3k+1$?
\end{prob}

If this is true, then it implies that $b(3k+1,3)=2k+2$. 
To see this, assume that the answer to the problem is affirmative, and
let vertices $v_1,v_2,\ldots,v_{2k+1}$ be given in $H(3k+1,3)$.
As in \cite{Alon}, we apply the existence of $w$ to 
$u_1:=v_2, u_2:=v_3,\ldots,u_{2k}:=v_{2k+1}$. Then we have
$g(v_i,w)<i-1$ for $i=2,3,\ldots,2k+1$.
Thus all of $w,w'$ and $w''$ have distance at least $2k+2-i$ from $v_i$
for all $2\leq i\leq 2k+1$.
Finally we use the fact that one of $w,w'$, and $w''$ has distance at least
$2k+1$ from $v_1$ because $d(v_1,w)+d(v_1,w')+d(v_1,w'')=2(3k+1)$.
Consequently, there is $w^*\in\{w,w',w''\}$ such that 
$d(v_i,w^*)\geq 2k+2-i$ for all $1\leq i\leq 2k+1$, that is,
$b(3k+1,3)\geq 2k+2$.

\section*{Acknowledgments}
The author thanks Noga Alon for valuable comments, in particular, for 
pointing out an error in the conjecture in the earlier version, 
Hajime Tanaka for providing information on eigenpolytopes and \cite{Godsil},
and Naoki Matsumoto for stimulating discussions.
He also thanks the referees who read the manuscript very carefully and 
corrected some errors.

The author was supported by JSPS KAKENHI Grant Number JP23K03201.

\end{document}